\documentclass[11pt,a4paper]{amsart}
\usepackage{amsmath,amsxtra,amsthm,amssymb,graphics,mathrsfs}
\usepackage[margin=2.5cm]{geometry}
\usepackage{enumerate}
\usepackage{stmaryrd}
\usepackage{indentfirst}
\usepackage{capt-of}
\usepackage{dsfont}
\usepackage{setspace}
\usepackage{caption}
\usepackage[font+=large]{subcaption}
\usepackage[all]{xy}
\usepackage{bbm}
\usepackage{graphicx}
\usepackage{xcolor}
\usepackage{verbatim}
\usepackage[colorlinks=true,linkcolor=black, anchorcolor=black,citecolor=black, urlcolor=black, bookmarksopenlevel=3]{hyperref}
\usepackage{wrapfig}

\newtheorem{theorem}{Theorem}[section]
\newtheorem{proposition}[theorem]{Proposition}
\newtheorem{lemma}[theorem]{Lemma}

\newtheorem{example}[theorem]{Example}
\newtheorem{definition}[theorem]{Definition}
\newtheorem*{proposition*}{Proposition}
\newtheorem{convention}[theorem]{Convention}
\newtheorem*{theorem*}{Theorem}
\newtheorem*{conjecture*}{Conjecture}

\theoremstyle{remark}
\newtheorem*{remark}{Remark}

\DeclareMathAlphabet{\mathpzc}{OT1}{pzc}{m}{it}
\DeclareMathOperator{\tensor}{\otimes}
\DeclareMathOperator{\isom}{\cong}

\DeclareMathOperator{\F}{\mathbb{F}}
\DeclareMathOperator{\FG}{\mathbb{F}\mathit{G}}

\DeclareMathOperator{\T}{\mathbbmss{T}}
\DeclareMathOperator{\Z}{\mathbb{Z}}

\DeclareMathOperator{\Hom}{\mathrm{Hom}}

\DeclareMathOperator{\st}{\mathrm{st}}

\DeclareMathOperator{\End}{\mathrm{End}}
\DeclareMathOperator{\Ext}{\mathrm{Ext}}
\DeclareMathOperator{\Soc}{\mathrm{soc}}

\DeclareMathOperator{\module}{\hyp\mathrm{mod}}

\DeclareMathOperator{\Top}{\mathrm{Top}}
\renewcommand{\S}{\mathbbmss{S}}

\mathchardef\hyp="2D

\makeatletter
\renewenvironment{proof}[1][\proofname]{\par
    \pushQED{\qed}%
    \normalfont \topsep0\p@\@plus6\p@ \labelsep1em\relax
    \trivlist
    \item[\hskip\labelsep\indent\bfseries #1]\ignorespaces
    \mbox{}
}{
    \popQED\endtrivlist\@endpefalse
}
\makeatother

\begin{document}

\begin{abstract}
	We will be defining a type of perverse equivalence that always corresponds to a derived equivalence with two-term tilting complexes. We are going to show that the tilting considered by Okuyama in \cite{Okuyama} and Yoshii in \cite{Yoshii} for the proof of Brou\'e's conjecture for $SL(2,q)$ in defining characteristic is a composition of such perverse equivalences.
\end{abstract}

\title{Perverse Equivalence in $SL(2,q)$}
\date{\today}
\author{William Wong}
\maketitle

\section{Introduction}

Brou\'e's Conjecture \cite{Broue}[6.2, Question] is a very important focal point of the block theory of finite groups:

\begin{conjecture*}
	Let $\F$ be the algebraically closed field of characteristic $p>0$. Let $G$ be a finite group and $A$ be a block of $\FG$. If $A$ has an abelian defect group $D$, then $A$ is derived equivalent to a block $B$ of $N_G(D)$, its Brauer correspondent.  
\end{conjecture*}

This conjecture has been studied by many and it is proved when $D$ is cyclic. However, for the general abelian case it is still a case-by-case verification. For $SL(2,q)$ in defining characteristic, the principal block of Brou\'e's conjecture is proven by Okuyama \cite{Okuyama} using a construction that differs from most of the other cases. Yoshii generalise Okuyama's method to the non-principal block case in \cite{Yoshii}. As time passes, tools such as mutation or perverse equivalence comes into play. In particular, perverse equivalence seems to gather some geometrical information of a certain derived equivalence in some surprising way. Some such example will be Craven's application of perverse equivalence with Lusztig's L-function in \cite{Craven}. This paper is to do the similar job for $SL(2,q)$, by showing the derived equivalence between full defect blocks of $SL(2,q)$ and its Brauer correspondent contemplated by Okuyama (and Yoshii) is a composition of perverse equivalences. As a consequence we extending the application of perverse equivalence to Okuyama's construction, under some condition on the projective modules of algebras involved.

In section 2 we introduce some known machinery needed for this article. This includes an introduction to the cochain complex Okuyama is utilising in his proof, perverse equivalence in total order (filtration) form and partial order form, and various information on the representation of $SL(2,q)$ we use for this paper. In section 3 we introduce a particular type of perverse equivalence which will yield two-term tilting complexes, but yet very poorly understood. Then we utilize some extra facts deduced by Okuyama in his approach to construct our string of perverse equivalence, arriving at the following theorem:

\begin{theorem*}[\ref{thm:main}]
	The derived equivalence between full defect blocks of $SL(2,q)$ in defining characteristics and its Brauer correspondent introduced by Okuyama and Yoshii are compositions of perverse equivalences.
\end{theorem*}

In section 4 we discuss some findings along the construction and we present $SL(2,9)$ as an example.

\section*{Acknowledgement}
The author thanks the support by JSPS International Research Fellowship during the development of the project.

\section{Background and preliminaries}\label{sec:Bkgd}

We use right modules for the exposition of this article in order to facilitate our reference to Okuyama's and Yoshii's paper. When $A$ is a block, let $A\module$ be the category of finitely-generated (fg) right $A$-modules, $\st(A)$ be the stable category of fg right $A$-modules and $D^b(A)$ be the bounded derived category of fg right $A$-modules. For complexes we use the cochain notation, that is, differential maps in a complex are of degree 1.

\subsection{Rickard-Okuyama tilting complex}\label{sec:OKTilt}

We introduce Rouquier's construction of a two-term bimodule complex and Okuyama's treatment to make such complex a tilting complex of symmetric finite-dimensional (fd) algebras. We start with a symmetric fd algebra $A$. Instead of just relying on information from $A$ we can utilise another symmetric fd algebra $B$ which is stably equivalent of Morita type. Using the simple-minded system in $A$ obtained from simple $B$-modules, under certain conditions this creates a tilting complex of $A$ with endomorphism algebra $C$ (which is symmetric fd). This algebra $C$, under some further conditions, would have all simple $C$-modules either 'inherited' from $A$ or from $B$ (we shall make this clear in the exposition). The hope is this process can be iterated until all simple $A$-modules have been replaced by simple $B$-modules. Then we can apply Linckelmann's theorem (quoted below) \cite{Linckelmann} to show that we have created a derived equivalence between $A$ and $B$. Effectively we have managed to 'lift' the stably equivalence of Morita type to derived equivalence between $A$ and $B$.

\begin{theorem}\label{thm:Linc}
	Let $A$ and $B$ be two self-injective algebras with no simple projective summands. If $A$ and $B$ are stably equivalent of Morita type which the equivalence sends simple $A$-modules to simple $B$-modules, then $A$ and $B$ are Morita equivalent.
\end{theorem}

Let us start be defining our terminology.
\begin{definition}
	Let $A$ and $B$ be two algebras that are stably equivalent of Morita type via ${}_BM_A$, a $(B,A)$-bimodule with ${}_BM$ and $M_A$ be projective $B^{op}$- and $A$-module respectively. Let $T_z$, $S_z$, $z \in Z$ be a complete set of mutually non-isomorphic simple $B$- and $A$-modules respectively. Let $\tau_z:Q_z \to T_z$ be a projective cover of $T_z$ and $\pi_z:P_z \to T_z \tensor_B M$ be a projective cover of $T_z \tensor_B M$ as an $A$-module. 
\end{definition}
\begin{remark}
	It is only conjectured that when $A$ and $B$ are stably equivalent (of Morita type) they should have the same number of simple modules (Auslander-Reiten Conjecture). We will assume this holds throughout the article (and it holds for all cases considered here).  
\end{remark}
Now there exists an $A$-homomorphism $\rho_z:P_z \to Q_z \tensor_B M$ such that $\pi_z=(\tau_z \tensor id_M)\circ \rho_z$ by the projectivity of $P_z$. Through the natural isomorphism \begin{equation*}
\Hom_A(P_z, Q_z \tensor_B M) \isom \Hom_{B^{op}\tensor A}(Q_z^*\tensor P_z, M)
\end{equation*} $\rho_z$ corresponds to a homomorphism $\delta_z:Q_z^*\tensor P_z \to M$. Then \[\bigoplus_{z\in Z}\delta_z:\bigoplus_{z \in Z}(Q_z^*\tensor P_z) \to M\] is a projective cover of $M$.  

Let $I$ be a fixed subset of $Z$ and define a cochain complex $M_{I}^{\bullet}$ of (B,A)-bimodules:
\begin{equation*}
M_{I}^{\bullet}=\big(\dots \to 0 \to \bigoplus_{i \in I}(Q_i^* \tensor P_i) \xrightarrow{\bigoplus_{i \in I}\delta_i} M \to 0 \to \dots\big).
\end{equation*} where $M$ is in degree 0. 

\begin{theorem}[1.1 \cite{Okuyama}, 2.1.1 \cite{Yoshii}]\label{thm:condAtilt}
	As a complex of projective $A$-modules, $M^{\bullet}_I$ is a tilting complex of $A$ if and only if the following conditions hold: For all $i \in I$, $j \notin I$, \begin{enumerate}
		\item $\Hom_A(T_j\tensor_B M, \Omega(T_i \tensor_B M))=0$.
		\item Any $A$-homomorphism from $P_i$ to $T_j \tensor_B M$ factors through $\pi$.
	\end{enumerate}
\end{theorem}
\begin{definition}\label{def:Okutilt}
	When the conditions in \ref{thm:condAtilt} hold we set an algebra \begin{equation*}
	C=\End_{D^b(A)}(M_{I}^{\bullet}). 
	\end{equation*} 
	We call such construction of a new algebra as an \emph{Okuyama tilt} of $A$ with respect to $(B, I)$.
\end{definition}
	The algebra $C$ is symmetric, finite-dimensional, with a left-$B$ action induced from $M^{\bullet}_{I}$, and is derived equivalent to $A$. Furthermore,
\begin{lemma}\label{lem:OY}
	There exists a direct summand $N^{\bullet}_{I}$ of $(C \tensor_B M_{I}^\bullet)$ as $(C,A)$-bimodule such that:\begin{enumerate}
		\item $N^{\bullet}_{I}$ is a split-endomorphism two-sided tilting complex for $(C,A)$.
		\item $N_I^\bullet \isom M_I^\bullet$ as the homotopy equivalent classes of $(B,A)$-bimodule complex. 
		\item $N^{\bullet}_{I}$ has the form of a two-term complex, i.e. $\big(\cdots \to 0 \to Q \xrightarrow{\delta}N \to 0 \to \cdots\big)$.
	\end{enumerate}  On the algebras $A$, $B$ and $C$: \begin{enumerate}
	\item $-\tensor_B C: B\module \to C\module$ and $\Hom_A(N,-):A\module \to C\module$ give stable equivalences (of Morita type).
	\item There is an algebra monomorphism from $B$ to $C$.
	\item $N$ has no projective summand as $(C,A)$-bimodule. 
	\item The algebra $C$ has no projective summands as a $(B,C)$-bimodule.
\end{enumerate}
\end{lemma}
By Okuyama, there are two ways to trace simple $C$-modules (\ref{prop:simAC} and \ref{prop:simBC}).
\begin{proposition}\cite[1.3]{Okuyama}\label{prop:simAC}
	Let $S$ be a simple $A$-module. If \[\Hom_A(S, T_i\tensor_B M)=0=\Hom_A(T_i\tensor_B M, S)\] for all $i \in I$, then $\Hom_A(N,S)\isom S\tensor_A N^*$ is a simple $C$-module. Furthermore, \begin{enumerate}
		\item $\Hom_{D^b(A)}(N^\bullet_{I},S)$ is the stalk complex $\Hom_A(N,S)$ concentrated in degree zero.
		\item $\Hom_A(N,S)\isom \Hom_A(M,S)$ as $B$-modules.
	\end{enumerate}
\end{proposition}
\begin{definition}\label{def:KwrtI}
	Define the set $K$ to be the subset of $Z$ such that either $\Hom_A(S_k, T_i \tensor_B M)$ or $\Hom_A(T_i \tensor_B M, S_k)$ is nonzero for some $i$. Or equivalently, \[Z\setminus K:=\{z\mid \Hom_A(S_z, T_i \tensor_B M)=0=\Hom_A(T_i \tensor_B M, S_z) \text{ for all }i \in I\}.\]
\end{definition}
So for a fixed $I\subset Z$ the set $K$ (depend on $I$) is defined such that $I \subset K \subset Z$. (In line with \cite{Okuyama} and \cite{Yoshii}.) Proposition \ref{prop:simAC} traced a correspondence of simple $A$-module indexed by $Z\setminus K$ with simple $C$-module (which we indexed via the correspondence). We need to obtain the rest of simple $C$-modules. The original Proposition applied in \cite{Okuyama} is as follows:
\begin{proposition*}
	\item For $i \in I$, if $dim_{\F}\Hom_A(\Omega(T_{i}\tensor_B M),\Omega(T_{i'}\tensor_B M))=\delta_{ii'}$ for all $i' \in I$, then $T_i \tensor_B M$ is a simple $C$-module.
	\item For $j \notin I$, if $dim_{\F}\Hom_A(T_{j}\tensor_B M,T_{j'}\tensor_B M)=\delta_{jj'}$ for all $j' \notin I$, then $T_j \tensor_B C$ is a simple $C$-module.
\end{proposition*}
However, the condition in (2) above is not fulfilled in some of our situations. We instead use the following variation:
\begin{proposition}\label{prop:simBC}With simple $C$-modules known from \ref{prop:simAC}, assume $K\setminus I$ has only one element, \begin{enumerate}
		\item For $i \in I$, if $dim_{\F}\Hom_A(\Omega(T_{i}\tensor_B M),\Omega(T_{i'}\tensor_B M))=\delta_{ii'}$ for all $i' \in I$, then $T_i \tensor_B M$ is a simple $C$-module.
		\item For (the only) $j \in K\setminus I$, if $dim_{\F}\Hom_A(T_{j}\tensor_B M,T_{j}\tensor_B M)=1$, then $T_j \tensor_B C$ is a simple $C$-module.
	\end{enumerate}
\end{proposition}
\begin{proof}
	Part (1) is exactly the same as \cite{Okuyama}[1.4] and we now check (2). There is only one remaining simple $C$-module to be found. Consider $j \in K\setminus I$, $z \in Z\setminus K$, we first check the homomorphisms \[\Hom_{C}(T_j\tensor_B C, \Hom_A(N,S_z)) \text{ and }\Hom_{C}(\Hom_A(N,S_z), T_j\tensor_B C)\]are zero. By tensor-hom adjunction, \begin{align*}
	\Hom_C(T_j\tensor_B C, \Hom_A(N,S_z))&\isom \Hom_A(T_j\tensor_B C \tensor_C N, S_z)\\
	&\isom \Hom_A(T_j\tensor M, S_z)=0, \text{ and }\\
	\Hom_{C}(\Hom_A(N,S_z), T_j\tensor_B C)&\isom \Hom_{C}(S_z \tensor_A N^*, T_j\tensor_B C)\\
	&\isom \Hom_A(S_z, \Hom_C(N^*, T_j\tensor_B C))\\
	&\isom \Hom_A(S_z, T_j\tensor_B C \tensor_C N))\\
	&\isom \Hom_A(S_z, T_j \tensor_B M)=0
	\end{align*} because of $z \in Z\setminus K$.
	So the top and socle of $T_j\tensor_B C$ is isomorphic to the missing simple $C$-module. Now the condition $dim_{\F}\Hom_A(T_{j}\tensor_B M,T_{j}\tensor_B M)=1$ shows there is only one copy of the said simple $C$-module. Hence $T_j\tensor_B C$ is the missing simple $C$-module we are looking for.
\end{proof}
In the situation we will encounter, all simple $C$-modules arises either from simple $A$-modules via \ref{prop:simAC} or from simple $B$-modules via \ref{prop:simBC}, respecting the derived equivalence and the stable equivalence respectively. This observation is the basis of applying theorem \ref{thm:Linc} to finish the proof for Okuyama and Yoshii. For this article, the same observation is the hint of applying alternating perverse equivalence, which we will define in \ref{def:simalt}.


\subsection{Perverse Equivalence}\label{sec:perveq}
To arrive at alternating perverse equivalence we first introduce perverse equivalence in general. The notion is introduced in \cite{ChuangRouquier}. A very brief summary would be 'equivalence of triangulated categories by shifted Serre subcategories'. For readers well-versed in mutation it can be understand as iterative simultaneous mutations with gradually smaller vertex set. We shall apply these to the equivalence between module categories of symmetric fd algebras. In this case, (assuming Krull-Schmidt property,) the Serre subcategories of such category are one-one correspondent with subset of simple modules. We first define some terminology.
\begin{definition}
	Let \begin{equation*}
	\emptyset = \S_{-1} \subset \S_0 \subset \dots \subset \S_n=\S
	\end{equation*}be a chain of subsets of the set of non-isomorphic simple modules $\S$. Denote by $\S^-_i$ the set $\S_i-\S_{i-1}$. 
	We say an element $S$ \emph{belongs to filtrate-$i$} if $S \in S^-_i$.
\end{definition}
We use the following definition for perverse equivalence, by Dreyfus-Schmidt in \cite{Leo}. 
\begin{definition}\label{def:perv}
	Let $A$ and $B$ be two symmetric finite dimensional algebras, \begin{equation*} \S_\bullet=(\emptyset=\S_{-1}\subset \S_0\subset \dots \subset \S_n=\S) \text{ and }\T_\bullet=(\emptyset=\T_{-1}\subset \T_0 \subset \dots \subset \T_n=\T) \end{equation*} be filtrations of the isomorphism class of simple $A$- and $B$-modules respectively. Let $\pi:\{0,\dots,r\} \to \Z$ be a function. An equivalence $F:D^b(A) \xrightarrow{\sim} D^b(B)$ is \emph{(filtered) perverse relative to $(\S_\bullet, \T_\bullet, \pi)$}, if for every $i$ with $0 \leq i \leq n$ the following holds. 
	\begin{itemize}
		\item Given $S \in \S^-_i$, the composition factors of $H^m(F(S))$ are in $\T_{i-1}$ for $m\neq -\pi(i)$ and there is a filtration $L_1 \subset L_2 \subset H^{-\pi(i)}(F(S))$ such that the composition factors of $L_1$ and of $H^{-\pi(i)}(F(S))/L_2$ are in $\T_{i-1}$ and those of $L_2/L_1$ are in $\T^-_i$.
		\item The map $S \to L_2/L_1$ induces a bijection $\S^-_i \xrightarrow{\sim} \T^-_i$. Hence there is an induced bijection of simple modules $\beta:\S \to \T$.
	\end{itemize}
\end{definition}

\begin{remark} The filtration $\S_\bullet$ and the induced bijection will decide the filtration $\T_\bullet$ (by $\beta(\S_i)$, $0 \leq i \leq n$), thus we will not explicitly write down $\T_\bullet$ thereafter. \end{remark}

Given an equivalence perverse relative to a certain filtration, the composition factors of the homology might belong to a smaller Serre subcategory. This correspond to a smaller subset of simple modules than the immediate subset in the filtration. This can be written into a partial order, coarser than the one given by filtration. If the homology of an equivalence is completely known, then the coarsest order available is 'the bare minimum of relations' required for this equivalence. To take benefit of this we introduce the notion of poset perverse equivalence, defined also in \cite{Leo}. 

\begin{definition}\label{def:poset}
	Let $A$, $B$ be two symmetric finite dimensional algebras, with $\S$ and $\S'$ the set of their non-isomorphic simple modules. A derived equivalence $F:D^b(A)\to D^b(A')$ is perverse relative to $(\S, \prec, \pi)$, where $(\S, \prec)$ is a poset structure on $\S$ and $\pi:\S \to \Z$, if and only if
	\begin{enumerate}
		\item There is a one-to-one correspondence $\beta_F:\S \to \S'$.
		\item Define $S_{\prec}=\{T \in \S \mid T\prec S\}$. The composition factors of $H^r(F(S))$ are in $\beta_F(S_{\prec})$ for $r\neq -\pi(S)$ and there is a filtration $L_1 \subset L_2 \subset H^{-\pi(S)}(F(S))$ such that the composition factors of $L_1$ and of $H^{-\pi(S)}(F(S))/L_2$ are in $\beta_F(S_{\prec})$ and $L_2/L_1$ is isomorphic to $\beta_F(S)$.
	\end{enumerate}
\end{definition}

The two notion introduced care about different aspects of a derived equivalence (that is perverse). The filtered perverse equivalence is mainly about existence, as given some filtration and function one can certainly create an equivalence perverse relative to that filtration and function. This does not hold for poset perverse equivalence, which emphasis the actual interaction of Serre subcategories that is being involved in a known equivalence. We list out some properties of both notion. The merits we mentioned can be seen in \ref{prop:filtperv}(5) and \ref{prop:poperv}(3) respectively.

\begin{proposition}\label{prop:filtperv} Let $F: D^b(A) \to D^b(B)$ be filtered perverse relative to $(\S_\bullet, \S'_\bullet, \pi)$.  \begin{enumerate}
		\item (reversibility) $F^{-1}$ is perverse relative to $(\S'_\bullet, \S_\bullet, -\pi)$.
		\item (composability) Let $F':D^b(B) \to D^b(C)$ be perverse relative to $(\S'_\bullet, \S''_\bullet, \pi')$, then $F' \circ F$ is perverse relative to $(\S_\bullet, \S''_\bullet, \pi+\pi')$.
		\item (refineability) Let $\tilde{\S}_\bullet=(0=\tilde{\S}_{-1} \subset \dots \subset \tilde{\S}_{\tilde{r}}$) be a refinement of $\S_\bullet$. Define the weakly increasing map $f:\{0,\dots,\tilde{r}\} \to \{0,\dots,r\}$ such that $\tilde{\S}_\bullet$ collapses to $\S_\bullet$ under $f$ (i.e. $\S_{f(i)-1} \subset \tilde{\S}_i \subset \S_{f(i)}$). Then $F$ is perverse relative to $(\tilde{\S}_\bullet, \pi\circ f)$. 
		\item If $\pi=0$ then $F$ restricts to a Morita equivalence of $A$ and $B$. 
		\item The information $(\S_\bullet, \pi)$ determines $B$ up to Morita equivalence.
\end{enumerate} \end{proposition}
\begin{proposition}\label{prop:poperv} Let $F: D^b(A) \to D^b(B)$ be poset perverse relative to $(\S, \prec, \pi)$.  \begin{enumerate}
		\item (reversibility) $F^{-1}$ is perverse relative to $(\beta(\S), \beta(\prec), -\pi)$.
		\item (composability) Let $F':D^b(B) \to D^b(C)$ be perverse relative to $(\beta(\S), \beta(\prec), \pi')$, then $F' \circ F$ is perverse relative to $(\S, \prec, \pi+\pi')$.
		\item (refineability) Let $\prec'$ be a finer partial order (i.e. $x \prec y$ induces $x \prec' y$). Then $F$ is perverse relative to $(\S, \prec', \pi)$. 
		\item If $\pi=0$ then $F$ restricts to a Morita equivalence of $A$ and $B$. 
\end{enumerate} where $\beta(\prec)$ is the partial order induced by $\beta$. That is, $\beta(S)\prec\beta(T)$ if and only if $S \prec T$. \end{proposition}

Note that given a derived equivalence which is a poset perverse equivalence, refining a partial order to a compatible total order we can obtain a filtered perverse equivalence. On the other hand, if more information on the homology of a derived equivalence is known, one can find a coarsest order to describe this equivalence. We give an example here to illustrate the definitions. In particular this is an example of \emph{elementary perverse equivalence}. See \cite{ChuangRouquier} for further details on this terminology.
\begin{example}[Elementary perverse equivalence in $SL(2,4)\isom A_5$, $p=2$]\label{ex:first} Let $A$ be the principal block of $kA_5$ with $k$ algebraically closed of characteristic 2. There are 3 non-isomorphic simple $A$-modules: The trivial module $k$, two modules $V$, $W$ which are two-dimensional. Their corresponding indecomposable projective covers have Loewy series as follows:
	\begin{equation*} P_k=\begin{matrix} k \\ V \quad W \\ k \quad k \\ W \quad V \\ k \end{matrix} \qquad \qquad P_V=\begin{matrix} V\\k\\W\\k\\V \end{matrix} \qquad \qquad P_W=\begin{matrix} W\\k\\V\\k\\W \end{matrix}. \end{equation*}
	In fact this is a Brauer graph algebra. We consider mutation at $\{V, W\}$. The summands of the tilting complex are given by: \[P_k \qquad \oplus \qquad P_V \to P_k \qquad \oplus \qquad P_W \to P_k\] where the rightmost term is in degree 0. It yields a correspondence of simple $A$-modules: \[k \mapsto \begin{matrix}k\\V\quad W\end{matrix}, \qquad V\mapsto V[1], \qquad W\mapsto W[1]\] where [1] is the shift (in $D^b(A\module)$). This derived equivalence is perverse with respect to the filtration $\big(\emptyset \subset \{V,W\} \subset \{k,V,W\}\big)$ and $\pi(i)=1-i$ for $i=0,1$. \\
\end{example}
By the image of simple $A$-modules (or the projective summands), define a partial order $\prec$ on $\{k,V,W\}$ by only $V \prec k$ and $W \prec k$. Also a function $\pi':k\mapsto 0, V\mapsto 1, W\mapsto 1$. Then this derived equivalence is also perverse with respect to $(\{k,V,W\}, \prec, \pi')$.
\begin{remark}
In fact, due to Rickard, the heart of the new $t$-structure introduced by the equivalence is isomorphic to $kA_4$-mod. 
\end{remark}
This is also the first (non-trivial) example of our (forthcoming) construction, where $A_5 \isom SL(2,4)$ and $A_4$ can be chosen as the normaliser of a Sylow-2 subgroup of $A_5$, which is the Klein-4 group. 

\subsection{Representation of SL(2,q)}

First we lay down some well-known facts for the representation theory of $SL(2,q)$ in defining characteristics. 

Let $G=SL(2,q)$, $q=p^n$, the base field $\F$ is an algebraically closed field of characteristic $p$. Let $V$ be the natural $G$-module, $V^{(r)}=\mathrm{Sym}^r(V)$ be the $r^{th}$ symmetric product of $V$. Let $\sigma$ be the Frobenius map on $\F$ and $V_m$ be the $m^{th}$ Frobenius twist of $V$. Then the tensor product \begin{equation*} 
S_z= \bigotimes_{i=0}^{n-1} V_i^{(z_i)} \end{equation*}
where $0 \leq z \leq q-1$ and $z=\sum_{i=1}^{n}z_i p^{i-1}$ written as $p$-adic number, form the complete set of non-isomorphic simple $\FG$-modules. 
(c.f. \cite[II 3.17]{Jantzen}) 
The simple $\FG$-modules fall into 3 blocks when $p$ is odd: \begin{enumerate}
	\item A defect zero block consisting of the simple module $S_{q-1}$.
	\item A (full defect) principal block consisting of simple modules $S_a$ with $a$ even.
	\item A (full defect) non-principal block consisting of simple modules $S_a$ with $a$ odd. \end{enumerate} 
When $p=2$ there are only two blocks, a defect zero block with $S_{q-1}$ and a full defect block with all the rest of simple modules. We take $A$ to be the direct sum of a copy each of the non-semisimple block(s). 

Take the Sylow $p$-subgroup $P$ of $G$ consisting of elements of the form $\begin{pmatrix} 1 & * \\ 0 & 1 \end{pmatrix}$. It is a defect group of $A$. We have $H=N_G(P)$ is the Borel subgroup $\begin{pmatrix} \alpha^{-1} & * \\ 0 & \alpha \end{pmatrix}$. Let $T_b$ be the one-dimensional $kH$-module with non-zero vector $t$ such that \begin{equation*}
t.\begin{pmatrix} \alpha^{-1} & * \\ 0 & \alpha \end{pmatrix}=\alpha^b(t). \end{equation*} These $T_b$ for $0 \leq b \leq q-2$ form a complete set of isomorphism classes of simple $kH$-modules (note $T_b \isom T_{q-1+b}$). 

If $p$ is odd, then $kH$ is split into two blocks: One with $T_b$'s of even subscript (principal block) and one with $T_b$'s of odd subscript (non-principal block). Again when $p=2$ these two blocks are merged into one. In any case we take the block $B$ to be $kH$, which is the Brauer correspondent of $A$. The restriction and induction between $A$-modules and $B$-modules induces a equivalence of their stable module categories. (This can be deduced from the Green correspondence. See, for example \cite[10.1]{Alperin}.) Okuyama and Yoshii then utilise some combinatorial properties of the composition factors of the induction and restriction, originated in \cite[6]{HSW}, to construct new algebras by tilting.


\subsection{Data for $SL(2,q)$}
Now we consider together Okuyama and Yoshii's proof of derived equivalence between $A$ and $B$. We paraphrase and combine their definitions here. First of all, we list out some conditions that is satisfied by some projective modules of the non-semisimple blocks of $SL(2,q)$ (and later some projective modules of all the intermediate algebras).

\begin{definition}\label{def:TPC}\cite{Okuyama}
	Given an algebra $A$, stably equivalent to a fixed algebra $B$ of Morita type via $M$, a $(B,A)$-bimodule. Let $S_z$ and $T_z$ ($z \in Z$) be the complete set of non-isomorphic simple modules of $A$ and $B$ respectively. $I \subset Z$ are some chosen index. We say the pair $(A, I)$ satisfies \emph{thin projective condition}, if the following holds for $A, I$ and $K$ (defined as \ref{def:KwrtI}): 
	\begin{enumerate}
		\item For $k \in K$, $T_k \tensor_B M$ is not simple, and \begin{enumerate}
			\item $\Soc(\Hom_A(M, S_k))\isom T_k \qquad \Top(\Hom_A(M, S_k))\isom T_{\tilde{k}}$.
			\item $\Soc(T_k \tensor_B M)\isom S_{\tilde{k}} \;\;\qquad \qquad \Top(T_k \tensor_B M)\isom S_k$.
		\end{enumerate}
		\item Let $P_z$ be the (minimal, indecomposable) projective cover of $S_z$, then
		\begin{enumerate}
			\item For $i \in I$, $\dim_{\F}\Hom_A(P_i, P_i)=2$ if $\tilde{i} \neq i$ \\
			\makebox[1.8cm]{}$\dim_{\F}\Hom_A(P_i, P_i)=3$ if $\tilde{i} = i$
			\item For $i, l \in I$, $\Hom_A(P_i, P_l)=0$ if $l\neq i$ and $l \neq \tilde{i}$.
			\item For $i \in I$, $\dim_{\F}\Hom_A(P_i, P_{\tilde{i}})=1$ if $\tilde{i} \neq i$.
		\end{enumerate}
		\item For $i \in I$, $\Hom_A(P_i, T_z\tensor_B M)=0$ if $z \notin K$.
	\end{enumerate}
\end{definition}

We called this 'thin projective condition' as in our case, the $P_i$, $i \in I$ are the ones with 'relatively few (and controllable) composition factors' in their Loewy layers.

\begin{definition}\label{def:orbits}
	Let $\S$ and $\T$ be the set of non-isomorphic simple $A$-modules and $B$-modules respectively. They share a common indexing set and we define such as $Z:=\{0,1,\ldots,q-2\}$. The \emph{Frobenius action} $\sigma$ on $z\in Z$ is \begin{equation*} \sigma(z)=pz \mod{q-1}. \end{equation*} and the \emph{sign action} $\tilde{.}$ on $z \in Z$ is 
	\begin{equation*} \tilde{z}=-z \mod{q-1}. \end{equation*} 	The actions allow us to define \emph{Frobenius orbits} and \emph{signed Frobenius orbits} of $Z$. 
\end{definition}

\begin{remark} The sign action and Frobenius action are commutative, hence each signed Frobenius orbits may split into two Frobenius orbits. We extend the use of $\sigma$ onto $Z$ since we have \[\sigma(S_z)=S_{pz}\text{ and }\sigma(T_z)= T_{pz}.\] The sign action comes from the dual of simple $B$-modules $T_z$: \[T^*_z:=\Hom_{k}(T_z,k)\isom T_{\tilde{z}}.\]
\end{remark}

\begin{definition}\label{def:IJK}
	Define an indexing on signed Frobenius orbits as follows: Let $K_{-1},\ldots,K_r$ be the partition of $Z$ into signed Frobenius orbits, such that \begin{enumerate}
		\item $K_{-1}=\{0\}$.
		\item $K_t$ is the set of all elements in a signed Frobenius orbit which contains the smallest number $z \notin \bigcup^{t-1}_{s=-1}K_s$. ($z \in K_t$) 
		\item $Z=\bigcup^r_{s=-1}K_s$.
	\end{enumerate}
	We define a Frobenius orbit $I_t \in K_t$, $0 \leq t \leq r$ such that the largest number in $K_t$ belongs to $I_t$. We also define $J_t=K_t\setminus I_t$.
	We denote by $K_{\leq t}:=\cup_{s=-1}^{t}K_s$ the union of $K_s$ for $-1 \leq s \leq t$. Likewise for $I$ and $J$.
\end{definition}
\begin{remark}
	$J_t$ is empty if and only if $z$ and $\tilde{z}$ is in the same Frobenius orbit $I_t$ for any $z \in I_t$. 
\end{remark}
\begin{example} Let $p=3$, $q=9=3^2$. Then the partition into signed Frobenius orbit is:
	\[K_{-1}=\{0\}\qquad K_0=\{1,3,5,7\} \qquad K_1=\{2,6\} \qquad K_2=\{4\}.\] The $I_t$, $0 \leq t \leq 2$ are \[I_0=\{7,5\}, \qquad I_1=\{6,2\}, \qquad I_2=\{4\}. \]\end{example}

Okuyama and Yoshii use these data to define successive algebras as Okuyama tilt with respect to $(B, I_t)$.  

\begin{definition}\label{def:OYString}
	Let $A$ be a (sum of) full defect block(s) of $kSL(2,q)$ and $B$ its (their) Brauer correspondent(s). Define algebras $A_t$ for $0 \leq t \leq r$ such that \begin{enumerate}
		\item $A_0=A$. 
		\item $A_{t}=End_{D^b(A_{t-1})}(M_{I_t}^{\bullet})$ for $1 \leq t \leq r$. 
	\end{enumerate}
\end{definition}
\begin{remark} For odd primes, each of the set $I_t$ contains only elements of the same parity. We regard the other algebra as unchanged. \end{remark}

The above definitions are well-defined since Okuyama and Yoshii show, in their respective papers, the following proposition:
\begin{proposition}\label{prop:OY}\cite{Okuyama}\cite{Yoshii}
	Let $t$ be any integer with $0 \leq t \leq r-1$. \begin{enumerate}
		\item $A_t$ is derived equivalent to $A$.
		\item There exists a unitary $k$-algebra monomorphism from $B$ to $A_t$.
		\item $A_t$ induces a stable equivalence of Morita type between $A_t$ and $B$. Moreover $A_t$ has no non-zero projective summands as $(B, A_t)$-bimodule.
		\item $A\tensor_B A_t$ is isomorphic to a direct sum of a nonprojective indecomposable module, denoted by $L_t$ and a projective module.
		\item Set \[S_{t,z}=\begin{cases}
		T_z \tensor_B A_t &\text{ for } z \in K_s\text{ with }s \leq t-1\\
		S_z \tensor_A L_t &\text{ for } z \in K_s\text{ with }s \geq t.\\
		\end{cases}\]Then $S_{t,z}$, $z \in Z$ is the complete set of mutually non-isomorphic simple $A_t$-modules.
		\item The pair $(A_t, I_t)$ satisfies thin projective condition. (c.f. definition \ref{def:TPC}).
	\end{enumerate}
\end{proposition}
Fact (6) further deduced that theorem \ref{thm:condAtilt} holds hence it is valid to set $A_{t+1}$ as an Okuyama tilt of $A_t$ with respect to $(B, I_t)$. (c.f. definition \ref{def:Okutilt})
 
To end this section, we remark that (5) gives a natural bijection between simple $A_{t_1}$-modules and simple $A_{t_2}$-modules ($0 \leq t_1, t_2 \leq r$) via their index (and their socles, in fact).

\section{Main theorem and construction}\label{sec:main}

In this section we first define simply alternating perverse equivalence (in \ref{def:simalt}) and characterise such equivalences. Then we show that, upon refining Okuyama's construction for $SL(2,q)$ to either the set $I$ has one element or $K\setminus I$ is empty, we have each of these Okuyama's tilt is simply alternating perverse. Then we compose these tilts to complete the proof of our main result (theorem \ref{thm:main}). 

\subsection{Perversity of two degrees}
Consider, for a certain filtration, the perverse function $\pi$ satisfies $\pi(i)=0$ when $i$ even and $\pi(i)=1$ when $i$ odd. We call this special case \emph{alternating perverse equivalence}. 
\begin{definition} \label{def:simalt}
	Let $A$ and $A'$ be algebras with a derived equivalence $F:D^b(A) \to D^b(A')$. Let $\S$ be the complete set of non-isomorphic simple modules of $A$. If there exists filtration $\S_\bullet=(\emptyset=\S_{-1}\subset \S_0 \subset \ldots \subset \S_r=\S)$ (of subsets of $\S$) and function $\pi$ that sends $i$ to $i$ modulo 2 for $0 \leq i \leq r$, then we say $F$ is an \emph{alternating} equivalence. In particular, if $r=2$ (hence $\S_\bullet=(\S_0 \subset \S_1 \subset \S_2=\S)$ and $\pi:\pi(0)=0; \pi(1)=1; \pi(2)=0$), then we say $F$ is a \emph{simply alternating} equivalence. Or, $A$ and $A'$ is \emph{simply alternating perverse equivalent} with respect to $(\S_\bullet, \pi)$.
\end{definition}
When $r=1$ this is just an elementary perverse equivalence. Usually the filtrations have to be reduced (i.e. $\S_i \neq \S_j$ for $i \neq j$). In this paper, for the convenience of presentation, when we consider simply alternating equivalence (i.e. $r=2$) we allow $\S_1=\S_2$, and set the following convention:
\begin{convention}\label{conv:degen}
	In the case $\S_1=\S_2$, the simply alternating perverse equivalence degenerates to an elementary perverse equivalence. That is, if we say $F$ is perverse with respect to filtration $\S_\bullet=\S_0 \subsetneq \S_1 = \S_2=\S$ and function $\pi:\pi(0)=0; \pi(1)=1; \pi(2)=0$, then we mean $F$ is perverse with respect to filtration $\S_\bullet=\S_0 \subsetneq \S_1=\S$ and function $\pi:i \mapsto i$ for $i=0,1$.
\end{convention}
	We will only discuss simply alternating in this section. This is partly due to the description of the image of simple modules under general alternating equivalence is very complex to describe. One can build any alternating equivalence by composing simply alternating ones, see section \ref{sec:discussion} for details.
For symmetric algebras, simply alternating equivalence is characterised by the following: 
\begin{theorem}\label{thm:simalt} Let $A$ and $A'$ be two symmetric algebras. Let $\S$ and $\S'$ be the sets of non-isomorphic simple $A$-modules and $A'$-modules, respectively. Let \[\S_\bullet=(\S_0 \subset \S_1 \subset \S_2=\S) \text{ and } \S'_\bullet=(\S'_0 \subset \S'_1 \subset \S'_2=\S')\] be filtrations of subsets of simple $A$-module and simple $A'$-module respectively. Suppose $F:D^b(A) \to D^b(A')$ is a simply alternating equivalence, with bijection of simple modules $S_z \leftrightarrow S'_z$. Then for $S_z \in \S^-_i$, $U'_z:=F(S_z) \in D^b(A')$ is:\begin{enumerate}
		\item $S'_z$ concentrated in degree 0 when $i=0$; 
		\item the largest quotient of $P'_z$ such that \begin{itemize} \item $\Top(U'_z)=S'_z$ and \item all other composition factors are in $\S'_0$, \end{itemize} concentrated in degree $1$, when $i=1$;
		\item the largest submodule $P'_z$ such that \begin{itemize} \item $\Soc(U'_z)=S'_z$. \item All other composition factors are in $\S'_1$. \item Composition factors of $\Top(U'_z)$ are in $\S'^-_1$. \end{itemize} concentrated in degree 0, when $i=2$.
	\end{enumerate} For $F^{-1}(S'_z)$, $S'_z \in \S'^-_i$ has image $U_z \in D^b(A)$ is\begin{enumerate}
		\item $S_z$ concentrated in degree 0 when $i=0$; 
		\item the largest submodule of $P_z$ such that \begin{itemize} \item $\Soc(U_z)=S_z$ and \item all other composition factors are in $\S_0$, \end{itemize} concentrated in degree $-1$, when $i=1$; ($U_z[-1]$ is the universal extension of $S_z$ by $\S_0$. See \cite{ChuangRouquier}.)
		\item the largest quotient of $P_z$ such that \begin{itemize} \item $\Top(U_z)=S_z$, \item All other composition factors are in $\S_1$, and \item Composition factors of $\Soc(U_z)$ are in $\S^-_1$, \end{itemize} concentrated in degree 0, when $i=2$. 
	\end{enumerate} 
\end{theorem}
\begin{remark} This perverse equivalence can be further decomposed as two (elementary) perverse equivalences: \begin{equation*}
	(\emptyset \subset \S_0 \subset \S),\qquad p_0:0 \mapsto 0, 1 \mapsto 1; \qquad \text{and} \qquad (\emptyset \subset \S_1 \subset \S), \qquad p_1:0 \mapsto 0, 1 \mapsto -1 \end{equation*}
	The results of the theorem is essentially computed from this.\end{remark}
\begin{proof}
	We will be using the inverse image later hence we shall just prove it for descriptions concerning $U_z$.
	First we prove the set $\{U_z\}$ is a simple-minded collection in $D^b(A\module)$ using Rickard's criterion, see \cite{Rickard3}.
	\begin{enumerate}
		\item $\Hom_{D^b(A\module)}(U_x, U_y)=\delta_{xy}$: \begin{enumerate}
			\item If $S_x$ and $S_y$ belongs to same even(resp. odd)-numbered partition $\S^-_i$, then the top (resp. socle) of $U_x$ and $U_y$ has a unique composition factor $S_x$ and $S_y$. All other composition factors of $U_x$ and $U_y$ belongs to $\S_{i-1}$. Hence when $x \neq y$ there is no non-zero map since the factor $S_x\in \S^-_i$ does not exist in $U_y$ or vice versa. When $x=y$ it must induced an isomorphic map. 
			\item  When $S_x \in \S^-_i$ is not in the same partition as $S_y \in \S^-_j$. Either \begin{itemize}
				\item both $i$ and $j$ are even, then $\{i,j\}=\{0,2\}$.
				\item either of $i$, $j$ is odd but not both. 
			\end{itemize} For the first case, take $U=F(\bar{S})$ for the $\bar{S}$ in partition-2. Neither the top nor socle of $U$ have composition factors in $\S_0$. Hence given the $\tilde{S}$ in partition-0, we have $\Hom_A(U, \tilde{S})=0=\Hom_A(\tilde{S},U)$. For the second case, if there is a non-zero map, then it is only possible when $i$ is even and $j$ is odd (because of the degree of $U_x$ and $U_y$). Hence we have $j=1$. We consider the injective resolution $I^*_y$ of $U_y$, in particular the term $I^0_y$. If $i=0$, since $U_y$ is the universal extension of $\S_0$ at degree $-1$, $\Soc(I^0_y)$ have no composition factor in $\S_0$. Thus, no non-zero maps exist from $U_x=S_x \in \S_0$ to $I^0_y$. If $i=2$ and a non-zero map $f:U_x \to U_y$ in $D^b(A\module)$ exists, we have a non-split short exact sequence \begin{equation*} 	0 \to U_y[-1] \to E \to U_x \to 0 \end{equation*} in $A\module$. This is because $U_y[-1]$ is a stalk complex at degree 0 hence we have isomorphism $\Hom_{D^b(A)}(U_x, U_y) \isom \Ext^1_A(U_x, U_y[-1])$. Let $E'$ be the largest module with composition factor only in $\S_0$ such that there exist non-split extension \begin{equation*} 0 \to E'' \to E \to E' \to 0 \end{equation*} in $A\module$.  
			Since $\Hom_A(E', U_x)=0=\Hom_A(U_x, E')$ ($E'$ has no composition factor in $\S_1^-$ and $\S_2^-$), the composite map $\varepsilon:U_y \to E \to E'$ maps surjectively into $E'$ while it is not bijective. Then we have (as picture) a non-split extension $0 \to \ker(\varepsilon) \to E'' \xrightarrow{\varepsilon} U_x$ with $\ker\varepsilon$ non-zero. Hence $E''$, while larger than $U_x$, also satisfy the description of (3), a contradiction to $U_x$'s assumed maximality.    
		\end{enumerate}
		\item $\Hom_{D^b(A')}(U_x, U_y[l])=0$ for negative integer $l$: A non-zero map is only possible when $l=-1$ and $S_y \in \S^-_1$. When $S_x \in \S^-_0$ the socle of $U_y$ is not in $\S'_0$ hence no non-zero map. When $S_x \in \S^-_2$ there is no composition factor belongs to $\S_2$ in $U_y$ hence again no non-zero map is possible. 
		\item The set of all $U_x$ generates $D^b(A')$ as a triangulated category: It is obvious that we can generate all non-isomorphic simple $A'$-modules of each layer by triangulating $U_x \in \S'^-_i$ with the simple factors in $\S'_{i-1}$.
	\end{enumerate}
	It is easy to check the generating criterion since every simple module can be iterated from the image with triangles involving lower filtration. Thus $U_x$'s are the image of simple modules under a derived equivalence. Since the composition factors of $U_x$ satisfy Definition \ref{def:perv}, hence the equivalence is perverse with function as given. 
\end{proof}

\begin{remark}Note $\{U_x\mid S_x \in \S_1^-\}$ is a right-finite semibrick and $\{U_x\mid S_x \in \S\setminus\S_1^-\}$ is a left-finite semibrick.\end{remark}

We also describe the summands of the tilting complex of a simply alternating perverse equivalences. 
\begin{theorem}\label{thm:altpervproj}
	Retain the notation of theorem (\ref{thm:simalt}). Denote the minimal (indecomposable) projective cover of $S_z$, $z \in Z$ by $P_z$. Then the tilting complex $T$ of $A$ such that $\End_{D^b(A)}(T)$ is Morita equivalent to $A'$ is given by the following summands for each $z$:
	\begin{enumerate}
		\item $P_{S_1^-} \to P_z$ for $z \in S_0$.
		\item $P_z \to P_{S_2^-}$ for $z \in S_1^-$.
		\item $0 \to P_z$ for $z \in S_2^-$.
	\end{enumerate}
	where the second term is in degree 0, $P_X$ is a direct sum of some $P_x$'s for $x \in X$.
\end{theorem}
\begin{remark} The summand $P_z \to P_{S_1^-}$ in (1) has no term involving $P_{S_2^-}$. Also, the term indicated by $P_{S_1^-}$ is not (necessary) the full approximation of $S_z$ by the additive closure of $P_{S_1^-}$. See Example \ref{ex:nonprincipal} for such a case. The precise description of the factor $P_{S_1^-}$ should be the 'right $P_{S_1^-}$-approximation of the left-$P_{S_2^-}$-approximated' factor of $S_z$.\end{remark}
We shall not prove this here since it is just an exercise of looking for the appropriate maps and cones in $K^b(\text{proj-}A)$ via left and right approximations.  


\subsection{A further breakdown of Okuyama's string on SL(2,q)}

Now we recall related notations from section 2.4. Our aim is to show, the algebra $A_t$ and $A_{t+1}$ in \cite{Okuyama} and \cite{Yoshii} is a composition of 'smaller' Okuyama's tilt. In this subsection we let $t$ be fixed.

\begin{definition}	Recall and continuing from definition \ref{def:IJK}. Fix $t$, we further partition the set $K_t$ into $K_{t,0},\ldots,K_{t,d}$ for some $d$ (depending on $t$), such that \[K_{t,c}:=\{z,\tilde{z}\} \text{ for the largest } z \in I_t \setminus (\bigcup_{b=0}^{c-1} K_{t,b})\text{ for } 0 \leq c \leq d\text{, with } K_t=\bigcup_{c=0}^{d}K_{t,c}.\] In other words, $K_{t,c}$ is the signed orbit that contains the largest number in $I_t$ yet to be partitioned.  Define the sets \[I_{t,c}=K_{t,c} \cap I_t \text{ and }J_{t,c}=K_{t,c}\setminus I_{t,c}.\] And define \[K_{\leq t,c}:=\bigcup_{s=1}^{t-1}K_s\cup\bigcup_{b=0}^{c}K_{t,b}.\] Likewise for $I_{\leq t,c}$ and $J_{\leq t,c}$.
\end{definition}

An example of this further partition of $Z$ is shown in \ref{ex:partition}. Under the definition the sets $K_{t,c}$, $0 \leq c \leq d$, will contain at most two elements each. Furthermore, $I_{t,c}$ is non-empty and $J_{t,c}$ has at most one element. 

\begin{definition}\label{def:WString}
	Fix $t$, Construct inductively a string of algebra $A_{t,c}$, $0 \leq c < d$, such that \begin{enumerate}
		\item $A_{t,0}=A_t$.
		\item $A_{t,c+1}=\End_{D^b(A_{t,c})}(M_{I_{t,c}}^\bullet)$. That is, $A_{t,c+1}$ is the Okuyama's tilt of $A_{t,c}$ with respect to $(B, I_{t,c})$.
	\end{enumerate}
\end{definition}

Now we need to show this is a well-defined definition. This part is essentially a repetition of the approach by Okuyama in section 3 of \cite{Okuyama} or section 4 of \cite{Yoshii}.
\begin{proposition}\label{prop:W}
	Fix an integer $t$. Let $c$ be any integer with $0 \leq c \leq d$. The algebra $A_{t,c}$ satisfies the following: \begin{enumerate}
		\item $A_{t,c}$ is derived equivalent to $A$.
		\item There exists a (unitary) $k$-algebra monomorphism from $B$ to $A_{t,c}$ and we have $A_{t,c} \isom B \oplus \text{ a projective (}B,B\text{)-bimodule}$. Hence $A_{t,c}$ induces stable equivalence of Morita type between $A_{t,c}$ and $B$. Furthermore $A_{t,c}$ as $(B,A_{t,c})$-bimodule has no nonzero projective summands.
		\item $A \tensor_B A_{t,c}$, as $(A, A_{t,c})$-bimodule, is isomorphic to a direct sum of a nonprojective indecomposable module, denoted by $L_{t,c}$, and a projective module. 
		\item Set \[S_{t,c,z}=\begin{cases}
		T_z \tensor_B A_{t,c}&\text{ if }z \in K_{\leq t, c-1}\\
		S_z \tensor_A L_{t,c}&\text{ if }z \in K_{\geq t, c}.
		\end{cases}\]Then the set of $S_{t,c,z}$ for all $z \in Z$ is the complete set of non-isomorphic simple $A_{t,c}$-modules.
		\item For $z \in J_{\leq t, c-1}$, every composition factor of $S_z \tensor_A L_{t,c}$ is isomorphic to $S_{t,c,y}$ for some $y \in K_{\leq t, c-1}$.
		\item $A_{t,c}$ and $I_{t,c}$ satisfies thin projective condition \ref{def:TPC}.
	\end{enumerate}
\end{proposition}
The proof is essentially a repetition of their respective proofs in loc. cit. with the appropriate alternation.
\begin{proof}
	For $c=0$: (1,2,3,4,5) is directly from loc. cit. since $A_{t,0}=A_t$. Recall that $K_{\leq t, -1}=K_{\leq t}$. (Similarly for $J$.) Since we have a smaller set of $I$ (and hence $K$), the only non-trivial-to-check condition is (3) (in \ref{def:TPC}) of (6). Let $l \in K_{t}\setminus K_{t,0}$ (i.e. $l \neq i, \tilde{i}$ for $i \in I_{t,0}$), if $l \in I_{t}$, consider a map in $\Hom_A(P_i, T_l \tensor_B M)$ must factor through the epimorphism $P_l \to T_l \tensor_B M$. However, $l\neq i,\tilde{i}$ hence $\Hom(P_i, P_l)=0$ (by (2b) in loc. cit.) forces $\Hom_A(P_i, T_l \tensor_B M)=0$. If $l \in J_{t}$, then a map in $\Hom_A(P_i, T_l \tensor_B M)$ extends to the injective hull of $T_l \tensor_B M$, which is $P_{\tilde{l}}$ for $\tilde{l} \in I_{t}$. Apply (2b) in loc. cit. to see they are zero. \\
	Now the induction part: (1,2,3) comes naturally from the theory of Okuyama (c.f. section \ref{sec:OKTilt}) (4,5,6) follows exactly as \cite{Yoshii}. 
\end{proof}

Lastly we have to show this string of equivalence is essentially the same as a step constructed by Okuyama and Yoshii. 
\begin{lemma}
	$A_{t,d+1}$ is Morita equivalent to $A_{t+1}$.
\end{lemma}
\begin{proof}
	Consider the set of simple $A_{t+1}$-modules is being given as \[S_{t+1,z}=\begin{cases}
	T_z \tensor_B A_{t+1}&\text{ for }z \in K_{\leq t+1}\\
	S_z \tensor_A L_{t+1}&\text{ for }z \in K_{> t+1}
	\end{cases}\]by \ref{prop:OY}. Apply the functor $- \tensor_{A_{t+1}}N^\bullet_{A_t}$ to see they are being correspondingly mapped from \[U_{t,z}:=\begin{cases}
	T_z \tensor_B A_{t}&\text{ for }z \in K_{\leq t+1}=K_{\leq t} \cup K_t\\
	S_z \tensor_A L_{t}&\text{ for }z \in K_{>t+1}
	\end{cases}\]in $A_t$. We note the fact $U_{t,z}$ is simple for $z \in K_{\leq t} \cup K_{>t+1}$, by \ref{prop:OY}. Now we send $U_{t,z}$ along the maps $- \tensor_{A_{t,c}} N_{I_{t,c}}^\bullet$ for $0 \leq c \leq d$ to $A_{t, d+1}$. By \ref{prop:W} all of these images of $U_{t,z}$ are simple $A_{t,d+1}$-modules. Hence we can conclude by \ref{thm:Linc} that $A_{t,d+1}$ and $A_{t+1}$ is Morita equivalent.
\end{proof}

In fact, consider each Okuyama's tilt with respect to $(B, I_{t,c})$ in our case is one and the same as Okuyama's tilt with respect to $(A_{t+1}, I_{t,c})$ we actually have an algebra isomorphism.


Now we are going to show that each of these 'smaller' Okuyama tilt is a perverse equivalence. In fact, these tilts are simply alternating equivalences.
\begin{lemma}\label{lem:main}
	The derived equivalence $F_{t,c}:D^b(A_{t,c}) \to D^b(A_{t,c+1})$ for $0 \leq t \leq r$, $0 \leq c < d$ is a perverse equivalence with respect to \begin{equation*}
	I_{t,c,\bullet}=(\emptyset = \S_{t,c,-1} \subset \S_{t,c,0} \subset \S_{t,c,1} \subset \S_{t,c,2} = \S_{t,c}), \qquad \pi(i)=\begin{cases}
	0 \text{ if } i \text{ is even} \\ -1 \text{ if } i \text{ is odd}.
	\end{cases}
	\end{equation*}where \begin{align*}
	&\S_{t,c,0}=\{S_{t,c,z}\mid z \in Z-K_{t,c}\}\\
	&\S_{t,c,1}=\{S_{t,c,z}\mid z \in Z-J_{t,c}\}\qquad &(\S_{t,c,1}^-=\{S_{t,c,z}\mid z \in I_{t,c}\})\\
	\end{align*}and when $J_{t,c}$ is empty, the equivalence degenerates as per \ref{conv:degen}. 
\end{lemma}
\begin{proof} 
	We shall be showing $F_{t,c}^{-1}$ is perverse with respect to $I_{t,c,\bullet}$ (defined on $A_{t, c+1}$) and $\pi:\pi(0)=0; \pi(1)=1; \pi(2)=0$. Consider the simple $A_{t,c+1}$-modules $S_{t,c+1,z}$, as stalk complexes in $D^b(A_{t, c+1})$ we have \[F_{t,c}^{-1}(S_{t,c+1,z})=\begin{cases}
	0 \to S_{t,c,z} &\text{ if }z \notin K_{t,c}\\
	P_{t,c,z} \to T_z \tensor_B M &\text{ if }z \in I_{t,c}\\
	0 \to T_z \tensor_B M &\text{ if }z \in J_{t,c}.
	\end{cases}\]with the second term in degree zero. Recall that $A_{t,c}$, $I_{t,c}$ satisfy thin projective condition \ref{def:TPC}. When $z \in I_{t,c}$, with (1b) of loc. cit., $H^{-1}(F_{t,c}^{-1}(S_{t,c+1,z}))$ has socle $S_{t,c,z}$. Given the structure of $P_{t,c,z}$ as in (2a), all composition factors of $H^{-1}(F_{t,c}^{-1}(S_{t,c+1,z}))/S_{t,c,z}$ is in $\S_{t,c,0}$. This concludes the case for a degenerated equivalence (i.e. if $J_{t,c}$ is empty, see \ref{conv:degen}). Otherwise, let $z \in J_{t,c}$ hence $z \neq \tilde{z}$ and $\tilde{z}\in I_{t,c}$. Consider $T_z \tensor_B M$ has injective envelope $P_{t,c,\tilde{z}}$. By (2a) there is only one copy of composition factor $S_{t,c,z}$ in $P_{t,c,\tilde{z}}$ hence the homology $H^0(F_{t,c}^{-1}(S_{t,c+1,z}))=T_z \tensor_B M$ contains this copy of composition factor $S_{t,c,z}$ as its top. Also by (2a) this shows $T_z \tensor_B M$ has no composition factors of either $S_z$ or $S_{\tilde{z}}$ except at top or socle. This concludes our argument since we have proved the homology of the required complexes satisfy definition \ref{def:perv} with the given data. 
\end{proof}
In fact, (continuing the notation above,) for $z \in I_{t,c}$, $H^{-1}(F_{t,c}^{-1}(S_{t,c+1,z}))$ must be the universal $\S_{t,c,0}$-extension by $S_{t,c,z}$ (or else yields a contradiction on $\Soc(T_z\tensor_B M)=T_{\tilde{z}}$). This satisfies (2) of theorem \ref{thm:simalt}. Furthermore, $T_z \tensor_B M$ must be isomorphic to the unique (indecomposable) module generated by $S_{t,c,z}$ and supported by $S_{t,c,\tilde{z}}$. This satisfies (3) of loc.cit.\\
Gathering all the tilts together we arrive at our main theorem.

\begin{theorem}\label{thm:main}
	Let $A$ be a full defect block of $kSL(2,q)$ in the defining characteristic. Then $A$ is derived equivalent to its Brauer correspondent $B$ via algebras $A_{t,c}$ $(0 \leq t \leq r, 0 \leq c \leq d(t))$ where $A_{t,c}$ and $A_{t,c+1}$ is simply alternating (or elementary) perverse equivalent for $0 \leq c \leq d$, and stringed up by $A_{0,0}=A$, $A_{r+1,0}=B$ and $A_{t, d+1}=A_{t+1,0}$ for $0 \leq t \leq r$.
\end{theorem}
\begin{proof} This follows from the induction we have in proposition \ref{prop:W}, lemma \ref{lem:main} and incorporating inductions in proposition \ref{prop:OY} by Okuyama and Yoshii. \end{proof}

\section{Example of SL(2,9) and compositions of alternating perverse equivalences}\label{sec:example}\label{sec:discussion}

This section demonstrates the case for $G=SL(2,9)$, $p=3$. After the example we make some observation about composing alternating equivalences.

\subsection{Example:SL(2,9)}
Recall this example has two full defect blocks, the subscript set is $Z=\{0,\dots,7\}$. We handle both blocks together hence the index will be slightly entangled. 
\begin{example}\label{ex:partition}
	The partition belong to the principal block $B_0$ is
	\begin{equation*}
	K_{-1}=\{0\} \text{, } K_1 = \{2, 6\} = I_1=J_1=K_{1,0} \text{ and } K_2 = \{4\}=I_2=J_2=K_{2,0}. \end{equation*}
	and for the non-principal block $B_1$ is 
	\begin{equation*}
	K_0 = \{1,3,5,7\} \text{, } J_0= \{1,3\} \text{ and } I_0=\{5,7\};   \end{equation*}  
	by which $K_0$ further split into 
	\begin{equation*}
	K_{0,0}=\{7,1\} \text{ and } K_{0,1}=\{5,3\}.
	\end{equation*}
	with 
	\begin{equation*}
	I_{0,0}=\{7\}, J_{0,0}=\{1\}; I_{0,1}=\{5\} \text{ and } J_{0,1}=\{3\}.
	\end{equation*}
\end{example}

\subsubsection{The principal block of SL(2,9)}
In $B_0$ all the signed Frobenius orbits coincides with Frobenius orbits (not true for general principal blocks of $SL(2,q)$ in prime $p$). So in $B_0$ all simply alternating equivalence degenerates to elementary perverse equivalence. 
Each $K_1$ and $K_2$ only yields one sub-$K$-set ($K_{1,0}=K_1$ and $K_{2,0}=K_2$), thus each Okuyama's tilt (using set $K_1$ and $K_2$) is perverse by itself.

\begin{example}\label{ex:principal}
	(Continuing from above) The set $K_{-1}$ needs no handling (in line with \cite{Okuyama}). Now (using information of projective modules computed from \cite{Koshita}) the perverse equivalence corresponding to the set $K_1$ gives one-sided tilting summands \begin{equation*}
	Q_0= P_2 \oplus P_6 \to P_0 \qquad Q_2 =  P_2 \to 0 \qquad Q_4=P_2 \oplus P_6 \to P_4 \qquad Q_6= P_6 \to 0
	\end{equation*}
	where the last term is in degree 0. The set $K_2$ further gives \begin{equation*}
	R_0=Q_4 \oplus Q_4 \to Q_0 \qquad R_2 = Q_4 \to Q_2 \qquad R_4=Q_4 \to 0 \qquad R_6=Q_4 \to Q_6
	\end{equation*}
	which compose into
	\begin{equation*}\begin{aligned}
	R_0= P_2 \oplus P_6 \to P_4 \oplus P_4 \to P_0 & \qquad R_2 = P_6 \to P_4 \to 0 \qquad \\ R_4=P_2 \oplus P_6  \to P_4 \to 0 & \qquad R_6= P_2 \to P_4 \to 0	
	\end{aligned}
	\end{equation*}	
\end{example}

\subsubsection{The non-principal block of SL(2,9)}
We will see the algorithm in full working in this case. The non-principal block has only one signed Frobenius orbit. It is also the smallest case that Proposition \ref{prop:intra-t} does not hold for at least one orbit. 

\begin{example}\label{ex:nonprincipal}
	We first produce the perverse equivalence with respect to the data given by $K_{0,0}$. Now using data from \cite{Koshita} the perverse equivalence gives one-sided tilting summands \begin{equation*}
	Q_1=0 \to P_1 \qquad Q_3 = P_7 \to P_3 \qquad Q_5=0 \to P_5 \qquad Q_7=P_7 \to P_1
	\end{equation*}
	where the last term is in degree 0. Note that in $Q_3$, $\Hom_{B_1}(P_7, P_3)$ is two-dimensional, but the dimension induced from $\Hom_{B_1}(P_1, P_3)$ is not included. The set $K_{0,1}$ further gives \begin{equation*}
	R_1=Q_5\to Q_1 \qquad R_3 = 0 \to Q_3 \qquad R_5=Q_5 \to Q_3 \qquad R_7=Q_5 \to Q_7
	\end{equation*}
	which yields (combined with Example \ref{ex:principal}) all summands of a one-sided tilting complex $T$ of $A$ where $\End_{D^b(A)}(T) \isom B$.
	\begin{equation*}
	R_1 =P_5 \to P_1 \qquad R_3 = P_7 \to P_3 \qquad R_5= P_5 \oplus P_7 \to P_3 \qquad R_7= P_5 \oplus P_7 \to P_1	
	\end{equation*}	
\end{example}

\subsection{Composing alternating perverse equivalence}

One main question is the possibility to compose perverse equivalences and the composition remain perverse. Some of these are inspired by the fact Okuyama can combine some of his tilting further \cite{Okuyama}. However such compositions in terms of perverse equivalence is not promising. In fact, Example \ref{ex:nonprincipal} is such bad example. The best environment to discuss these is the notion of poset perverse equivalence, introduced in section \ref{sec:perveq}. So from the perspective of composition factors, the question is whether the universal extensions (such as those consider in \ref{thm:simalt}) contains certain factors or not.
First we define a mechanism to turn inclusion of sets into partial orders.
\begin{definition}\
	Given a filtration/inclusion of sets $I:(\emptyset=\S_{-1} \subset \S_0 \subset \S_1 \subset \ldots \subset \S_r=\S)$ on $\S$, we define a partial order on the set of $\S$, denoted by $\prec_{\S,I}$ by imposing the filtration but no other relations. That is, given elements $S \in \S_i\setminus \S_{i-1}$, $S' \in \S_{i'}\setminus \S_{i'-1}$, \[S' \prec_{\S, I} S \text{ if and only if } i'<i.\]
\end{definition}
	We give a sufficient condition for $F_t$ to be a perverse equivalence for a fixed $t$. 
\begin{proposition}\label{prop:intra-t}
	Using notation as in section \ref{sec:main}. If for $z \in I_t$, the $A_t$-module $(\Omega T_z \tensor_B M)/S_{t,z}$ does not contain any composition factor of $S_{t,z}$ for $z \in J_t$ then all $F_{t,c}$, $0 \leq c \leq d$ composes to one simply alternating perverse equivalence. This one simply alternating perverse equivalence is naturally isomorphic to $F_t$.
\end{proposition}
\begin{proof}
	Translating filtered perverse equivalence to poset perverse equivalence, as stated in section \ref{sec:perveq}, $F_{t,c}$ is a poset perverse equivalence with respect to $(\S_{t,c}, \prec_{\S_{t,c}, I_{t,c}}, \pi')$, where $\pi'$ is the function sending $S_{t,c,z} \in \S_{t,c,i}^-$ to $\pi(i)$. Then by definition, the homology $H^{-1}(F_{t,c}(S_{t,c,z}))$ for $z \in I_{t,c}$ and $H^0(F_{t,c}(S_{t,c,y}))$ for $y \in J_{t,c}$ does not involve any composition factors $S_{t,c,y'}$ with $y' \in J_t\setminus J_{t,c}$. Hence $F_{t,c}$ is also a poset perverse equivalence with respect to $(\S_{t,c}, \prec_{\S_t, I^\bullet_t}, \pi')$ for all $0 \leq c \leq d$. As all $F_{t,c}$ is a poset perverse equivalence with respect to this same partial order $\prec_{\S_t, I^\bullet_t}$, we can compose them into one perverse equivalence. This equivalence is naturally isomorphic to $F_t$ since they possesses the same partial order and perverse function. Now translating to filtered perverse equivalence we have the (naturally expected) filtration coming from the inclusion $Z\setminus K_t \subset Z \setminus I_t \subset Z$ with the perverse function simply alternating. 
\end{proof} 
	We have mentioned in Section 3 that we can create a general alternating perverse equivalence using a composition of simple alternating ones. Consider any alternating perverse equivalence $F$ with a filtration of $r$ layers. This can be broken down as a composition of elementary perverse equivalence $F_i$ where each is perverse relative to \begin{equation*}
(\emptyset=\S_{-1} \subset \S_i \subset \S, p:0 \to 0, 1 \to (-1)^i).
\end{equation*}
for $1 \leq i \leq r$. (They are all elementary, in fact.) Then we can see that $E_j:=F_{2j} \circ F_{2j-1}$ with $1 \leq j \leq \lfloor r+1/2 \rfloor$ is a simple alternating perverse equivalence such that the composition of $E_j$ is naturally isomorphic to $F$ (define $F_{r+1}$ to be the identity if $r$ is odd to make $E_j$ well-defined, or use degeneration convention such as \ref{conv:degen}). We use this to contemplate composition of equivalences among steps of Okuyama's tilt (i.e. different $t$'s). 
\begin{proposition}
	Let $F_s$, $F_{s+1}$,...,$F_{s'}$ be a string of Okuyama tilt such that 
	for any $t$, $s \leq t \leq s'$, the $A_t$-module $\Omega (T_z \tensor_B M)/S_{t,z}$, $z \in I_t$, has no composition factor isomorphic to $S_{t,y}$ for all $y \in J_{t'}$ with $t \leq t'$.
	Then all $F_s$, $F_{s+1}$,...,$F_{s'}$ composes into one alternating perverse equivalence.
\end{proposition}
\begin{proof}
	Note that setting $t'=t$ for the condition will make $\F_t$ satisfy the condition of \ref{prop:intra-t}, hence the condition secures each $F_t$ is a simply alternating perverse equivalence. Consider the set inclusion \[I^{\bullet}_{s,s'}=\big(Z\setminus\bigcup_{t=s}^{s'}K_t \subset Z\setminus(\bigcup_{t=s}^{s'}K_t\setminus I_s)\subset Z\setminus\bigcup_{t=s+1}^{s'}K_t\subset\ldots\subset Z\big)\] and the hence defined partial order $\prec_{Z, I^{\bullet}_{s,s'}}$. For each $t$, the condition deduces that the composition factors involved in the homology of $H^{-1}(F_t(S_{t,z}))$ for $z \in I_t$ and $H^0(F_t(S_{t,y}))$ for $y \in J_t$ does not involve $S_{t, y}$ for any $y \in J_{t'}$ for all $t' \geq t$. So $F_t$ is a poset perverse equivalence with respect to \[\big( \S_t, \prec_{\S_t, I^\bullet_{s,s'}}, \pi_t\big),\]where $\prec_{\S_t, I^\bullet_{s,s'}}$ is defined by transferring the partial order on $Z$ to modules $\S_t$ via indexing, and $\pi_t(S_{t,z})=\begin{cases}
			1 &\text{ if } z \in I_t\\ 0 &\text{ otherwise}.
			\end{cases}$. Thus, all $F_t$, $s \leq t \leq s'$ composes into one alternate perverse equivalence. 
\end{proof}
	Not every Okuyama's tilt can be expressed as an alternate perverse equivalence. In particular, example \ref{ex:nonprincipal} does not satisfy any proposition above. We know by the non-existence of stalk projective summands there is no way the example is any kind of perverse equivalence. Though it should be obvious to careful readers the projective summands are very 'perverse-alike' in the sense all $P_5$ and $P_7$ are concentrated in degree $-1$ and $P_1$ and $P_3$ in degree 0. Although we have shown for the case $SL(2,q)$ in this paper, we do not know if every Okuyama's tilt can be undoubtedly written as a composition of perverse equivalence. Our proof depends on thin projective condition, but it is an unnecessary strong requirement in considering the general case. This might be an interesting topic related to the geometry aspect of Okuyama's tilting.

We end this section and the paper by giving an affirmative answer to a comment by Okuyama \cite[p.15]{Okuyama} in his first paper introducing his tilting complex. His method will lead to Chuang's \cite{Chuang} and Holloway's \cite{Holloway} constructions of derived equivalences in $SL(2,p^2)$ (in fact, the dual result) after a long but trivial calculation.


\end{document}